\documentclass[12pt]{article}
\usepackage{mathrsfs}
\usepackage{amscd}
\usepackage{amsmath,amsfonts,amssymb,amscd,amsthm}
\usepackage{indentfirst,graphics,epsfig,psfrag}
\usepackage{ifpdf}
\usepackage{enumerate}
\usepackage{appendix}
\usepackage{enumerate}
\usepackage{color}
\usepackage{lineno}
\usepackage{cite}
\newtheorem{thm}{Theorem}[section]
\newtheorem{cor}{Corollary}[section]
\newtheorem{lem}{Lemma}[section]

\newtheorem{obs}{Observation}[section]
\newtheorem{pro}{Proposition}[section]

\newcommand{\mpo}{\mbox{mp}}
\title{\textbf{Matching preclusion number of graphs}
    \footnote{Supported by the National Science Foundation of China
        (Nos. 11601254, 11551001, 11161037, and 11461054) and the Science
        Found of Qinghai Province (Nos.  2016-ZJ-948Q, and 2014-ZJ-907).}}
\author{
\small Zhao Wang$^{1}$, \ \ Yaping Mao$^{2,5}$\footnote{Corresponding author}, \ \  Eddie Cheng$^{3}$, \ \ Jinyu Zou$^{4}$\\[0.1cm]
\small $^1$College of Science, China Jiliang\\
\small University, Hangzhou 310018, China\\[0.1cm]
\small $^2$School of Mathematics and Statistics, Qinghai\\
\small Normal University, Xining, Qinghai 810008, China\\[0.1cm]
\small $^3$Department of Mathematics and Statistics,\\
\small Oakland University, Rochester, MI USA 48309\\[0.1cm]
\small $^4$School of Computer Sciences, Qinghai\\
\small Normal University, Qinghai 810008, China\\[0.1cm]
\small $^5$Center for Mathematics and Interdisciplinary Sciences\\
\small  of Qinghai Province, Xining, Qinghai 810008, China\\[0.1cm]
\small E-mails: wangzhao@mail.bnu.edu.cn; maoyaping@ymail.com; \\
\small ~~~~  echeng@oakland.edu; zjydjy2015@126.com}
\date{}
\begin{document}

\maketitle
\begin{abstract}
The \emph{matching preclusion number}
of a graph $G$, denoted by $\mpo(G)$, is the minimum number of edges whose deletion
results in a graph that has neither perfect matchings nor
almost-perfect matchings. In this paper, we first give some sharp
upper and lower bounds of matching preclusion number. Next, graphs with large and
small matching preclusion number are characterized, respectively. In the end,
we investigate some extremal problems and the Nordhaus-Gaddum-type relations on matching preclusion number. \\[2mm]
\end{abstract}

\section{Introduction}
In a parallel machine, the underlying topology is a graph where vertices are processors and edges are physical links between processors. Such a graph in an interconnection network. One may want to cluster the processors
into groups and the most basic clusters are clusters of two processors. Such a clustering may or may not be possible if there are failures in some of the links. So it may be desirable to consider interconnection networks that are
resilient to faulty links under such a clustering requirement.
 
All graphs considered in this paper are undirected, finite and
simple. We refer to the book \cite{Bondy} for graph theoretical
notation and terminology not described here. For a graph $G$, let $V(G)$, $E(G)$,
and $\overline{G}$ denote the set of vertices, the set
of edges, and the complement of $G$, respectively. The number of vertices in $G$ is the \emph{order} of $G$.
For any subset $X$ of $V(G)$, let $G[X]$ denote the subgraph induced
by $X$; similarly, for any subset
$F$ of $E(G)$, let $G[F]$ denote the subgraph induced by $F$. Let $X\subseteq V(G)\cup E(G)$. We use
$G-X$ to denote the subgraph of $G$ obtained by removing all the
vertices in $X$ together with the edges incident with them from $G$ as well as
removing all the edges in $X$ from $G$. If $X=\{x\}$, we may write $G-x$ instead of
$G=\{x\}$.
For two subsets $X$ and $Y$ of
$V(G)$ we denote by $E_G[X,Y]$ the set of edges of $G$ with one end
in $X$ and the other end in $Y$.

The {\it degree}\index{degree} of a vertex $v$ in a graph $G$,
denoted by $deg_G(v)$, is the number of edges of $G$ incident with
$v$. Let $\delta(G)$ and $\Delta(G)$ be the minimum degree and
maximum degree of the vertices of $G$, respectively. The set of
neighbors of a vertex $v$ in a graph $G$ is denoted by $N_G(v)$.
A graph is
\emph{Hamiltonian} if it contains a Hamiltonian cycle. A component of a
graph is \emph{odd} or \emph{even} according to whether it has an
odd or even number of vertices.

A \emph{matching} in a graph is a set of edges such that every vertex is incident with at most one edge in this set. If a set of edges form a matching in a graph, they are \emph{independent}.
A \emph{perfect matching} in a graph is a set of edges such that every vertex is incident with exactly one edge in this set. An
\emph{almost-perfect matching} in a graph is a set of edges such that every vertex, except one, is incident with exactly one edge in
this set, and the exceptional vertex is incident to none. So if a graph has a perfect matching, then it has an even number of
vertices; if a graph has an almost-perfect matching, then it has an odd number of vertices. The \emph{matching preclusion number}
of a graph $G$, denoted by $\mpo(G)$, is the minimum number of edges whose deletion leaves the resulting graph with neither
perfect matchings nor almost-perfect matchings. Such an optimal set is called an \emph{optimal matching preclusion set}. We define
$\mpo(G)=0$ if $G$ has neither perfect matchings nor almost-perfect matchings. This concept of matching preclusion was
introduced in \cite{BrighamHVY} and further studied in \cite{BrighamHVY, ChengLJ, ChengLJa, ChengCM, ChengL, ChengL2, ChengHJL, ChengLLL, Park, ParkS, ParkI, ParkI2, WangWLL, MWCM, WangFZ}. Originally this concept was introduced as a measure of robustness in the event
of edge failure in interconnection networks. It is worth nothing that besides this application, it was also remarked in \cite{BrighamHVY} that this measure has a theoretical connection to other concepts in graph theory such as conditional connectivity and extremal graph theory.

The following results are immediate.
\begin{obs}{\upshape \cite{ParkI}}\label{obs1-1}
$(1)$ If $H$ is a spanning subgraph of $G$, then
$\mpo(H)\leq \mpo(G)$.

$(2)$ For an even graph $G$, $\mpo(G)\leq \delta(G)$.

$(3)$ If $e$ is an edge of $G$, then $\mpo(G-e)\geq \mpo(G)-1$.
\end{obs}

As mentioned earlier, the topologies of the underlying architecture of distributed processor systems or parallel machines are important since good topologies offers the advantages of improved connectivity and reliability. We refer the readers to \cite{HsuL} for recent progress
in this area and the references in its extensive bibliography. For application that requires clustering, the most clusters are of size two. This is exactly the concepts of perfect matchings. Moreover,
the matching preclusion number measures the robustness under this clustering requirement in the event of link
failures, as indicated in \cite{BrighamHVY}. Naturally, we want the matching preclusion number to be high.

Let $\mathcal {G}(n)$ denote the class of simple graphs of order $n$. Given a graph theoretic parameter $f(G)$ and a positive
integer $n$, the \emph{Nordhaus-Gaddum Problem} is to
determine sharp bounds for: $(1)$ $f(G)+f(\overline{G})$ and $(2)$
$f(G)\cdot f(\overline{G})$, as $G$ ranges over the class $\mathcal
{G}(n)$, and characterize the extremal graphs. We consider the following problems in which their solutions will give insights in designing interconnection networks with respect to the size of the networks and the targeted matching preclusion number.

\noindent {\bf Problem 1.} Given two positive integers $n$ and $k$,
compute the minimum integer $s(n,k)=\min\{|E(G)|:G\in
\mathscr{G}(n,k)\}$, where $\mathscr{G}(n,k)$ the set of all graphs
of order $n$ (that is, with $n$ vertices) with matching preclusion number $k$.


\noindent {\bf Problem 2.} Given two positive integers $n$ and $k$,
compute the minimum integer $f(n,k)$ such that for every connected
graph $G$ of order $n$, if $|E(G)|\geq f(n,k)$ then $\mpo(G)\geq k$.

\noindent {\bf Problem 3.} Given two positive integers $n$ and $k$,
compute the maximum integer $g(n,k)$ such that for every
graph $G$ of order $n$, if $|E(G)|\leq g(n,k)$ then $\mpo(G)\leq k$.

We remark that while Problem 1 and Problem 3 are over all graphs, Problem 2 is over all connected graphs. Problem 2 would not be a good problem if it is over all graphs as we want to find the smallest $f(n,k)$ to guarantee the matching
preclusion number of a graph on $n$ vertices and $f(n,k)$ edges is at least $k$. Since the graph with two singletons and a $K_{n-2}$ has matching preclusion number 0, the bound on the number of edges would be too big, so it is not a good question to consider.

In Section 2, we give sharp upper bounds of $\mpo(G)$ for a graph $G$ and show that
$$
0\leq \mpo(G)\leq \begin{cases}
n-1 & \text{if $n$ is even,}\\
2n-3 & \text{if $n$ is odd.}
\end{cases}
$$
for a graph $G$ of order $n$. In Section 3, graphs with $\mpo(G)=0,1$, even graphs with $\mpo(G)=n-k \ (n\geq 4k+6)$, and odd graphs with $\mpo(G)=2n-3,2n-4,2n-5$ are characterized, respectively. In Section 4, we study the above extremal problems on matching preclusion number.
The Nordhaus-Gaddum-type results on matching preclusion number are given in Section 5. Many researchers found Nordhaus-Gaddum-type results aid their understanding of network vulnerability parameters such as edge-toughness, scattering number and bandwidth.
Indeed this type of results are of interest to both graph theorists and network analysts. See \cite{Aouchiche} for a survey.

\section{Sharp bounds on matching preclusion number}

From complete graphs, Brigham et al. \cite{BrighamHVY} derived the following
result.
\begin{thm}{\upshape \cite{BrighamHVY}}\label{lem2-1}
For a complete graph $K_n$,
$$
\mpo(K_n)=\begin{cases}
n-1 & \text{if $n$ is even and $n\geq 2$,}\\
2n-3 & \text{if $n$ is odd and $n\geq 9$.}
\end{cases}
$$
\end{thm}

The following result, due to Dirac, is well-known.
\begin{thm}{\upshape Dirac \cite{Bondy} \ (p-485)}\label{thA}
Let $G$ be a simple graph of order $n \ (n\geq 3)$ and minimum
degree $\delta$. If $\delta\geq \frac{n}{2}$, then $G$ is
Hamiltonian.
\end{thm}

If $G$ is an even graph (that is, $G$ has an even number of vertices), then $\mpo(G)\leq \delta(G)$.
From Theorem~\ref{lem2-1}, this bound is not true for odd graphs. We now consider
upper bounds of $\mpo(G)$ for an odd graph $G$.

Given an edge $uv\in E(G)$, the
parameter $\xi_G(uv)=d_G(u)+d_G(v)-2$ is \emph{the degree of the
edge $uv$} and the parameter $\xi(G)=\min\{\xi_G(uv)\,|\, uv \in
E\}$ is the \emph{minimum edge-degree} of $G$.

\begin{pro}\label{pro2-1}
Let $G$ be an odd graph. Then
$$
\mpo(G)\leq \xi(G)+1.
$$
Moreover, the bound is sharp.
\end{pro}
\begin{proof}
From the definition of $\xi(G)$, there exists two vertices $u,v$
such that $\xi_G(uv)=\xi(G)$.  Let $X=E_G[u,N_G(u)]\cup
E_G[v,N_G(v)]\cup \{uv\}$. Since $u$ and $v$ are two isolated vertices in
$G-X$, it follows that $G-X$ contains neither perfect matching nor
almost perfect matching, and hence $\mpo(G)\leq
|X|=d_G(u)+d_G(v)-1=\xi(G)+1$.

To show the sharpness of this bound, we consider the complete graph
$K_n$, where $n$ is odd. By Theorem~\ref{lem2-1}, $\mpo(K_n)=2n-3=\xi(K_n)+1$ for $n\geq 9$.
\end{proof}

We now give a slightly more sophisticated upper bound.

\begin{pro}\label{pro2-3}
Let $G$ be an odd graph, and let $v$ be a vertex of $G$ such that
$d_G(v)=\delta(G)$. Then
$$
\mpo(G)\leq \delta(G)+\delta(G-v).
$$
Moreover, the bound is sharp.
\end{pro}
\begin{proof}
Let $u$ be a vertex of degree $\delta(G-v)$ in $G-v$. Let
$X=E_G[v,N_G(v)]\cup E_G[u,N_G(u)]$. Clearly, $u,v$ are two isolated
vertices in $G-X$, and hence there is no almost-perfect matchings in
$G-X$. So $\mpo(G)\leq d_G(v)+d_{G-v}(u)=\delta(G)+\delta(G-v)$, as
desired.

To show the sharpness of this bound, we consider the path graph $P_n
\ (n\geq 3)$, where $n$ is odd. Clearly,
$\mpo(P_n)=2=\delta(G)+\delta(G-v)$. Another example is to start with a $K_{2n+1}$ (with $n\geq 4$) and attach two pendant vertices. Then the resulting graph has matching preclusion number 2. (This also shows that this bound is
better than the one in Proposition~\ref{pro2-1}.
\end{proof}

Note that each graph $G$ with $n$ vertices is a spanning subgraph of
$K_n$. The following bounds are immediate by Observation
\ref{obs1-1} and Theorem~\ref{lem2-1}.

\begin{pro}\label{pro2-2}
Let $G$ be a connected graph of order $n$. Then
$$
0\leq \mpo(G)\leq \begin{cases}
n-1 & \text{if $n$ is even and $n\geq 2$,}\\
2n-3 & \text{if $n$ is odd and $n\geq 9$.}
\end{cases}
$$
Moreover, the bounds are sharp.
\end{pro}

\section{Graphs with given matching preclusion numbers}

In this section, we characterize graphs with large and small matching preclusion numbers.

\subsection{Graphs with small matching preclusion numbers}

Let $o(G)$ be the number of odd components of $G$. If a graph $G$
has a perfect matching $M$, then $o(G-S)\leq |S|$ for all
$S\subseteq V(G)$. The following result due to Tutte shows that the converse is
true.

\begin{thm}{\upshape \cite{Tutte}}\label{thC}
A graph $G$ has a perfect matching if and only if every subset $S$
of vertices satisfies Tutte's condition, that is, $o(G-S)\leq |S|$.
\end{thm}

Berge \cite{Berge} obtained a similar result for almost-perfect matchings.

\begin{thm}{\upshape \cite{Berge}}\label{thD}
A graph $G$ has an almost-perfect matching if and only if every subset $S$ of vertices  satisfies Berge's condition, that is, $o(G-S)\leq |S|+1$.
\end{thm}

Clearly the above theorems can be used to characterized graph with matching preclusion number at most $k$ in the following way for even graphs: $\mpo(G)\leq k$ if and only if there exist $T\subset E(G)$ where
$|T|\leq k$ and $S\subseteq V(G-T)$ such that $o(G-S\cup T)>|S|$. Similarly for odd graphs. From an algorithmic point of view, the theorems of Tutte and Berge are important because such conditions formed a basis for a polynomial time algorithm in determining whether a graph has a perfect matching and almost perfect matching. Thus if $k$ is fixed, then there is a polynomial time algorithm to determine whether $\mpo(G)\leq k$. If $k$ is not fixed, then the problem of determining whether
$\mpo(G)\leq k$ is {\cal NP}-complete; see \cite{LacroixMMP}.

The following simple result gives a bound on the smallest degree if $\mpo(G)=k$.

\begin{pro}\label{pro3-2}
Let $G$ be a graph of order $n$ and $k\geq 1$. If $\mpo(G)=k$, then $\delta(G)\leq \frac{n}{2}+k-2$.
\end{pro}
\begin{proof}
Suppose $\delta(G)\geq \frac{n}{2}+k-1)$. Let $F$ be an optimal matching preclusion set. Let $e\in F$. Let Then $\delta(G-(F\setminus\{e\}))\geq \frac{n}{2})$. From Theorem \ref{thA}, $G-(F\setminus\{e\})$
contains a Hamiltonian cycle. Thus $G-F$ contains a
perfect matching or an almost-perfect matching, which is a contradiction.
\end{proof}

\subsection{Graphs with large matching preclusion number}

We first characterize even graphs.

\begin{pro}\label{pro3-4}
Let $G$ be an even graph of order $n\geq 2$. Then $\mpo(G)=n-1$
if and only if $G$ is a complete graph.
\end{pro}
\begin{proof}
From Theorem~\ref{lem2-1}, if $G$ is a complete graph of order $n$, then $\mpo(G)=n-1$. Conversely, we suppose $\mpo(G)=n-1$. Then $\delta(G)\geq \mpo(G)=n-1$, and hence $G$ is a complete graph, as desired.
\end{proof}

\begin{thm}\label{th3-3}
Let $G$ be an even graph of order $n\geq 4$. Then $\mpo(G)=n-2$
if and only if $\delta(G)=n-2$.
\end{thm}
\begin{proof}
If $\mpo(G)=n-2$, then $\delta(G)\geq \mpo(G)=n-2$, and hence
$\delta(G)=n-2$ by Proposition \ref{pro3-4}. Conversely, if
$\delta(G)=n-2$, then $\mpo(G)\leq n-2$. We need to show $\mpo(G)\geq
n-2$. It suffices to prove that for any $X\subseteq E(G)$ and $|X|=n-3$,
$G-X$ has a perfect matching. Since $\delta(G)=n-2$, it follows that
$G$ is a graph obtained from $K_{n}$ by deleting a perfect matching. Note
that there are $n-2$ edge-disjoint perfect matchings in $G$. Since we
only delete $n-3$ edges from $G-X$, it follows that $G-X$ has a
perfect matching. So $\mpo(G)\geq n-2$. From the above argument, we
conclude that $\mpo(G)=n-2$.
\end{proof}

For $n\geq 4k+6$, we have the following general result.
\begin{thm}\label{th3-5a}
Let $n,k$ be two integers with $n\geq 4k+6$, and let $G$ be an even
graph of order $n$. Then $\mpo(G)=n-k$ if and only if $\delta(G)=n-k$.
\end{thm}
\begin{proof}
Suppose $\delta(G)=n-k$. Then $\mpo(G)\leq n-k$. We need to show
$\mpo(G)\geq n-k$. It suffices to prove that for every $X\subseteq E(G)$ and
$|X|=n-k-1$, $G-X$ has a perfect matching. We
first suppose that $deg_{G[X]}(v)\leq \frac{n-2k}{2}$ for every $v\in
V(G)$. Then
$$
deg_{G-X}(v)=deg_{G}(v)-deg_{G[X]}(v)\geq (n-k)-\frac{n-2k}{2}=\frac{n}{2},
$$
and hence $\delta(G-X)\geq \frac{n}{2}$. From Theorem \ref{thA},
$G-X$ contains a Hamiltonian cycle, and hence there is a perfect
matching in $G-X$. Next, we suppose that there exists a vertex $v\in
V(G)$ such that $deg_{G[X]}(v)\geq \frac{n-2k+2}{2}$. Since
$deg_{G-X}(v)\geq deg_{G}(v)-|X|\geq 1$, it follows that there
exists a vertex $u\in V(G)$ such that $vu\in E(G-X)$. Let
$G_1=G-\{u,v\}$. Clearly, $|V(G_1)|=n-2$ is even, and $|X\cap
E(G_1)|\leq n-k-1-\frac{n-2k+2}{2}=\frac{n-4}{2}$. Since $|X\cap
E(G_1)|\leq \frac{n-4}{2}$, it follows that for any vertex pair
$s,t\in V(G_1)$, $deg_{G_1-X}(s)+deg_{G_1-X}(t)\geq
2(n-k-2)-\frac{n-4}{2}-1\geq n$ since $n\geq 4k+1$. So $G_1-X$ contains a Hamiltonian
cycle, and hence there is a perfect matching in $G_1-X$, say $M'$.
Clearly, $M'\cup \{uv\}$ is a perfect matching of $G-X$. From the
above argument, we conclude that $\mpo(G)=n-k$.

Conversely, we suppose $\mpo(G)=n-k$. We want to show that
$\delta(G)=n-k$. Furthermore, by induction on $k$, we prove that
$\mpo(G)=n-k$ if and only if $\delta(G)=n-k$. From Proposition
\ref{pro3-4}, Theorem \ref{th3-3}, the result
follows for $k=1,2$. Suppose that the argument is true for every
integer $k' \ (k'<k)$, that is, $\mpo(H)=n-k'$ if and only if
$\delta(H)=n-k'$. For integer $k$, it follows from Observation
\ref{obs1-1} that $\delta(G)\geq \mpo(G)=n-k$. We need to show that
$\delta(G)=n-k$. Assume, on the contrary, that $\delta(G)>n-k$. Let
$\delta(G)=n-k+t=n-(k-t)$, where $t\geq 1$. Since $k-t<k$, it
follows from the induction hypothesis that $\mpo(G)=n-k+t<n-k$, which
contradicts $\mpo(G)=n-k$. So $\delta(G)=n-k$.
\end{proof}

Next, we characterize odd graphs with $\mpo(G)=2n-3,2n-4,2n-5$, respectively.
\begin{pro}\label{pro3-5}
Let $G$ be an odd graph of order $n\geq 9$. Then $\mpo(G)=2n-3$ if and only if $G$ is a complete graph.
\end{pro}
\begin{proof}
From Theorem~\ref{lem2-1}, if $G$ is an odd complete graph of order
$n$, then $\mpo(G)=2n-3$. Conversely, we suppose $\mpo(G)=2n-3$. If $G$
is not a complete graph, then there exist an edge $e=uv\notin
E(G)$. Let $X=E_G[u,N_G(u)]\cup E_G[v,N_G(v)]$. Since $u,v$ are two
isolated vertices in $G-X$, it follows that $G-X$ has no
almost-perfect matchings, and hence $\mpo(G)\leq |X|\leq 2n-4$, a
contradiction. So $G$ is a complete graph, as desired.
\end{proof}

\begin{thm}\label{th3-5}
Let $G$ be an odd graph of order $n\geq 9$. Then $\mpo(G)=2n-4$ if and only if $G=K_n-e$, where $e\in E(K_n)$.
\end{thm}
\begin{proof}
If $\mpo(G)=2n-4$, then it follows from Proposition \ref{pro3-5} that
$\delta(G)\leq n-2$. We claim that $G=K_n-e$. Assume, on the
contrary, that $G\neq K_n-e$. Then $\overline{G}$ contains $P_3=vuw$
or two independent edges $xy,uv$ as its subgraph. For the former
case, let $X=E_G[u,N_G(u)]\cup E_G[v,N_G(v)]$, and $|X|\leq
2(n-2)-1=2n-5$. Since $u,v$ are two isolated vertices in $G-X$, it
follows that $\mpo(G)\leq |X|\leq 2n-5$. For the latter case, let
$Y=E_G[u,N_G(u)]\cup E_G[x,N_G(x)]$. Similarly, $\mpo(G)\leq |Y|\leq
2(n-2)+1-2=2n-5$, a contradiction. So $G=K_n-e$.

Conversely, if $G=K_n-e$, then $\mpo(G)\leq 2n-4$. We need to show $\mpo(G)\geq 2n-4$. It suffices to prove that for every $X\subseteq E(G)$ and $|X|=2n-5$, $G-X$ has an almost-perfect matching. Since $G=K_n-e$, it follows that $G-X=K_n-(X\cup \{e\})$ is a graph obtained from $K_{n}$ by deleting at most $2n-4$ edges. By Proposition \ref{pro2-2}, $G-X$ has an almost-perfect matching. So $\mpo(G)\geq 2n-4$, and together with the above argument, $\mpo(G)=2n-4$.
\end{proof}

\begin{thm}\label{th3-6}
Let $G$ be an odd graph of order $n\geq 13$. Then $\mpo(G)=2n-5$
if and only if $G$ satisfies one of the following conditions.

$(1)$ $\delta(G)=n-2$ and $G\neq K_n-e$;

$(2)$ $G=K_n-E(P_3)$.
\end{thm}
\begin{proof}
Suppose $\mpo(G)=2n-5$. Then we have the following claims.

\textbf{Claim 1.} $\delta(G)\geq n-3$.

\noindent \textbf{Proof of Claim 1.} Assume, on the contrary, that
$\delta(G)\leq n-4$. Then there exists a vertex $u$ in $G$ such that
$d_G(u)\leq n-4$. Pick $v\in N_G(u)$. Then $uv\in E(G)$. Let
$X=E_G[u,N_G(u)]\cup E_G[v,N_G(v)]$. Then $|X|\leq
(n-4)+(n-2)=2n-6$. Clearly, $u,v$ are two isolated vertices in
$G-X$, and there are no almost-perfect matchings in $G-X$. So
$\mpo(G)\leq |X|\leq 2n-6$, which is a contradiction. $\Diamond$

By Proposition \ref{pro3-5}, Theorem \ref{th3-5} and Claim 1, we
have $\delta(G)=n-2$ and $G\neq K_n-e$, or $\delta(G)= n-3$. Thus we may assume that
$\delta(G)=n-3$. Furthermore, we have
the following claim.

\textbf{Claim 2.} $G=K_n-E(P_3)$.

\noindent \textbf{Proof of Claim 2.} Assume, on the contrary, that
$G\neq K_n-E(P_3)$. Since $\delta(G)=n-3$, it follows that $G$ is a
spanning subgraph of $K_n-E(P_3)$ where $P_3=uvw$ and $v$ has degree 2 in $G$. Moreover, at least one additional edge
$e=xy$ is deleted. Clearly, $xy\neq uv$ and $xy \neq vw$. Suppose
$u=x$. Then $d_G(u),d_G(v)=n-3$. Let $X=E_G[u,N_G(u)]\cup
E_G[v,N_G(v)]$. Then $|X|\leq (n-3)+(n-3)=2n-6$.
Since $G-X$ has two isolated vertices,
there are no almost-perfect matchings in $G-X$. So $\mpo(G)\leq |X|\leq 2n-6$,
which is a contradiction. Suppose $x,y\notin \{u,w\}$. Let
$X=E_G[v,N_G(v)]\cup E_G[x,N_G(x)]$. Thus $|X|\leq n-3+n-2-1=2n-6$ since the edge $xv$ is incident to both $x$ and $v$. As before, $G-X$ has no
almost-perfect matchings since $v$ and $x$ are isolated. Then $\mpo(G)\leq |X|\leq 2n-6$, which is a
contradiction. $\Diamond$

In summary, we have $G=K_n-E(P_3)$, or $\delta(G)=n-2$ and $G\neq
K_n-e$, as required.

Conversely, if $G=K_n-E(P_3)$, then it follows from Proposition
\ref{pro3-5} and Theorem \ref{th3-5} that $\mpo(G)\leq 2n-5$. We need
to show $\mpo(G)\geq 2n-5$. It suffices to prove that for every $X\subseteq
E(G)$ and $|X|=2n-6$, $G-X$ has an almost-perfect matching. Since
$G=K_n-E(P_3)$, it follows that $G-X=K_n-(X\cup E(P_3))$ is a graph
obtained from $K_{n}$ by deleting at most $2n-4$ edges. By
Proposition \ref{pro3-5}, $G-X$ has an almost-perfect matching, as
desired.

Suppose that $\delta(G)=n-2$ and $G\neq K_n-e$.
Since $\delta(G)=n-2$, $G=K_n-L$ where $L$ is a matching of $K_n$ of size at least 2.
By Proposition
\ref{pro3-5} and Theorem \ref{th3-5}, $\mpo(G)\leq 2n-5$. We
need to show $\mpo(G)\geq 2n-5$. It suffices to prove that for every
$X\subseteq E(G)$ and $|X|=2n-6$, $G-X$ has an almost-perfect
matching. If $G-X$ has an isolated vertex, say $v$, then $n-2\leq
|X\cap E_G[v,N_G(v)]|\leq n-1$, and hence $|X\cap E(G-v)|\leq n-4$.
We note that there are $n-2$ edge-disjoint perfect matchings in $K_{n-1}$.
In fact, we can say that given a perfect matching $N$ in $K_{n-1}$, the edges $K_{n-1}-N$ can be decomposed into $n-3$ edge-disjoint
perfect matchings. We observed earlier that $G=K_n-L$ where $L$ is a matching of $K_n$ of size at least 2. Thus $G-v=K_{n-1}-L'$ where $L'$
is a matching of $K_{n-1}$. Extend $L'$ to $N$, a perfect matching of $K_{n-1}$. Now $G-v-N$ has
$n-3$ edge-disjoint
perfect matchings, and hence $G-v-N-X$ contains at least
$n-3-(n-4)=1$ perfect matching, say $M'$. Clearly, $M'$ is a perfect matching of $G-v-X$ and $M'$ is
an almost-perfect matching of $G-X$ missing $v$.

From now on, we may assume that $G-X$ has no isolated vertices. If
$deg_{G[X]}(v)\leq \frac{n-5}{2}$ for every $v\in V(G)$, then
$$
deg_{G-X}(v)= deg_{G}(v)-deg_{G[X]}(v)\geq (n-2)-\frac{n-5}{2}=\frac{n+1}{2}>\frac{n}{2},
$$
and hence $\delta(G-X)> \frac{n}{2}$. By Theorem \ref{thA}, $G-X$
contains a Hamiltonian cycle, and hence there is an almost-perfect
matching in $G-X$, as desired.

Suppose that there exists a vertex $v\in V(G)$ such that
$deg_{G[X]}(v)\geq \frac{n-3}{2}$. Since $G-X$ has no isolated
vertex, it follows that there exists a vertex $u\in V(G)$ such that
$uv\in E(G-X)$. Let $G_1=G-\{u,v\}$. Note that $|V(G_1)|=n-2$ is odd,
and $|X\cap E(G_1)|\leq 2n-6-\frac{n-3}{2}=\frac{3n-9}{2}$. If
$G_1-X$ has an isolated vertex, say $w$, then $n-4\leq |X\cap
E_{G_1}[w,N_{G_1}(w)]|\leq n-3$, and hence $|X\cap E(G_1-w)|\leq
\frac{3n-9}{2}-(n-4)=\frac{n-1}{2}$.
Recall that $G=K_n-L$ where $L$ is a matching of $K_n$ of size at least 2.
Let $G_2=G-\{u,v,w\}=K_n-L-\{u,v,w\}$. Let $T$ be the (possibly empty) matching of $K_n-\{u,v,w\}$ induced from $T$. Now $K_n-\{u,v,w\}$ has $n-4$ disjoint perfect matchings, one of which contains $T$.
Thus $G_2=G-\{u,v,w\}=K_n-L-\{u,v,w\}$ has at least $n-5$ disjoint perfect matchings. Since $n\geq 13$, $n-5-(n-1)/2 \geq 1$.
Thus
$G-X-\{u,v,w\}$ contains a perfect matching, say $M'$.
Clearly, $M'\cup \{uv\}$ is an almost-perfect matching of
$G-X$ missing $w$.

From now on, we may assume that $G_1-X$ has no isolated vertices. If
$deg_{G_1[X]}(x)\leq \frac{n-9}{2}$ for every $x\in V(G_1)$, then
$$
deg_{G_1-X}(x)\geq deg_{G_1}(x)-deg_{G_1[X]}(x)\geq (n-4)-\frac{n-9}{2}=\frac{n+1}{2}>\frac{n}{2},
$$
and hence $\delta(G_1-X)> \frac{n}{2}$. By Theorem \ref{thA}, $G_1-X$ contains a Hamiltonian cycle, and hence there is an almost-perfect matching in $G_1-X$, say $M'$. Clearly, $M'\cup \{uv\}$ is an almost-perfect matching of $G-X$ missing $w$.

Suppose that there exists a vertex $s\in V(G_1)$ such that
$deg_{G_1[X]}(s)\geq \frac{n-7}{2}$. Since $G_1$ has no isolated
vertices, it follows that there exists a vertex $t\in V(G_1)$ such
that $st\in E(G_1-X)$. Let $G_2=G_1-\{s,t\}$. Note that $|V(G_2)|=n-4$
is odd, and $|X\cap E(G_2)|\leq
2n-6-\frac{n-3}{2}-\frac{n-7}{2}=n-1$.

Observe that $G_2$ is a graph
from $K_{n-4}$ by deleting at most $\frac{n-5}{2}$ edges. Then
$G_2-X$ is a graph from $K_{n-4}$ deleted at most
$\frac{n-5}{2}+n-1=\frac{3n-7}{2}< 2(n-4)-3$ edges. By Theorem~\ref{lem2-1}, there is an almost-perfect matching in $G_2-X$, say
$M'$. (Here we require $n-4\geq 9$, which is satisfied as $n\geq 13$.) Clearly, $M'\cup \{uv,st\}$ is an almost-perfect matching of
$G-X$ missing $w$.

We may now conclude that $\mpo(G)=2n-5$.
\end{proof}

We remark in the above proof, there is one place that we requires $n\geq 13$. It is in the second last paragraph when we apply Theorem~\ref{lem2-1}. Although $G_2$ is a graph
from $K_{n-4}$ by deleting at most $\frac{n-5}{2}$ edges, these edges are independent.  Thus it may be possible to exploit this structure to replace the 13 in the $n\geq 13$ requirement to a smaller number. We feel that it is not
worthwhile to lengthen this discussion by this potential marginal improvement.

\section{Extremal problems on matching preclusion number}

We now consider the three extremal problems that we stated in the Introduction. We first give the results for $s(n,k)$.

\begin{lem}\label{lem4-1}
Let $n,k$ be two positive integers such that $n\geq 3$ is odd. Then

$(1)$ $s(n,0)=0$;

$(2)$ $s(n,1)=\frac{n-1}{2}$;

$(3)$ $s(n,2)=n-1$;

$(4)$ $s(n,3)=n$.
\end{lem}
\begin{proof}
$(1)$ Let $H_1$ be the graph of order $n$ with no edges. Clearly,
$\mpo(H_1)=0$. Then $s(n,0)\leq 0$, and so $s(n,0)=0$.

$(2)$ Let $H_2$ be a graph of order $n$ with $\frac{n-1}{2}$
independent edges. Clearly, $\mpo(H_2)=1$ and $H_2$ has $\frac{n-1}{2}$ edges.
Then $s(n,1)\leq \frac{n-1}{2}$. Conversely, let $G$ be an odd graph
of order $n$ such that $\mpo(G)=1$. Since $G$ contains an
almost-perfect matching, it follows that $G$ has at least $\frac{n-1}{2}$ edges,
and hence $s(n,1)\geq \frac{n-1}{2}$. So $s(n,1)=\frac{n-1}{2}$.

$(3)$ Let $H_3$ be a path of order $n$. Since $n$ is odd, it follows
that $\mpo(H_3)=2$ and $H_3$ has $n-1$ edges. Then $s(n,2)\leq n-1$.
We now prove that this inequality holds as equality.

Assume, on the contrary, that
$s(n,2)\leq n-2$. Then there exists an odd graph $G$ of order $n$
with $s(n,2)\leq n-2$ edges
such that $\mpo(G)=2$. Clearly, $G$ is not
connected. Let $C_1,C_2,\ldots,C_r$ be the connected components in
$G$. If two of $C_1,C_2,\ldots,C_r$ are odd components, then
$\mpo(G)=0$, which is a contradiction. So there is at most one odd component
in $G$. Since $|V(G)|$ is odd, it follows that there is exactly one
odd component in $G$. Without loss of generality, we may assume that
$|V(C_1)|$ is odd, and $|V(C_i)|$ is even for $2\leq i\leq r$. We
claim that $\delta(C_i)\geq 2$ for $2\leq i\leq r$. Assume, on the
contrary, that there exists a component $C_j$ such that
$\delta(C_j)=1$. Then there exists a vertex $v$ such that
$d_G(v)=1$. Let $uv$ be the pendent edge of $C_j$. Clearly, $v$ and
$C_1$ are two odd components of $G-uv$, which contradicts the fact
that $\mpo(G)=2$. So $\delta(C_i)\geq 2$ for $2\leq i\leq r$. Then
$|E(G)|=\sum_{i=1}^r|E(C_i)|\geq
(|V(C_1)|-1)+\sum_{i=2}^r|V(C_i)|=n-1$, which is a contradiction.

$(4)$ Let $H_4$ be a cycle of order $n$. Since $n$ is odd, it
follows that $\mpo(H_4)=3$ and $|E(H_4)|=n$. Then $s(n,3)\leq n$.
We now prove that this inequality holds as equality.

Assume, on the contrary, that
$s(n,3)\leq n-1$. Since $s(n,2)=n-1$, it follows that $s(n,3)\geq
s(n,2)=n-1$, and hence $s(n,3)=n-1$. Then there exists an odd graph
$G$ of order $n$ such that $\mpo(G)=3$ and $|E(G)|=s(n,3)=n-1$. If $G$
is connected, then $G$ is a tree, and hence $\mpo(G)=2$, which is a
contradiction. We now assume that $G$ is not connected. Let
$C_1,C_2,\ldots,C_r$ be the connected components in $G$. If two of
$C_1,C_2,\ldots,C_r$ are odd components, then $\mpo(G)=0$, which is a
contradiction. So there is at most one odd component in $G$. Since
$|V(G)|$ is odd, it follows that there is exactly one odd component in
$G$. Without loss of generality, we may assume that $|V(C_1)|$ is odd,
and $|V(C_i)|$ is even for $2\leq i\leq r$. Furthermore, we have the
following claim.

\textbf{Claim 1.} $\delta(C_i)\geq 3$ for $2\leq i\leq r$.

\noindent \textbf{Proof of Claim 1.} Assume, on the contrary, that
there exists a component $C_j$ such that $\delta(C_j)\leq 2$ for
$2\leq i\leq r$. Then there exists a vertex $v$ in $C_j$ such that
$d_G(v)\leq 2$. Let $X=E_G[v,N_G(v)]$. Clearly, $v$ and $C_1$ are
two odd components of $G-X$, which contradicts the fact that
$\mpo(G)=3$ as $|X|\leq 2$. $\Diamond$

From Claim 1, we have $\delta(C_i)\geq 3$ for $2\leq i\leq r$. Then
$2|E(G)|=2\sum_{i=1}^r|E(C_i)|\geq 2(|V(C_1)|-1)+3\sum_{i=2}^r|V(C_i)|=2(|V(C_1)|-1)+3(n-|V(C_1)|)$, and hence $|E(G)|\geq n+\frac{n-|V(C_1)|-2}{2}\geq n$, which is a contradiction.
\end{proof}

\begin{thm}\label{pro4-3}
Let $n,k$ be two positive integers. Then

$(1)$ If $n\geq 2$ is even and $0\leq k\leq n-1$, then
$s(n,k)=\frac{nk}{2}$.

$(2)$ If $n\geq 5$ is odd and $4\leq k\leq 2n-6$, then
$$
\frac{n(n-1)k}{4n-6} \leq s(n,k)\leq
\min\left\{\left\lceil\frac{k}{3}\right\rceil,\frac{n-1}{2}\right\}n.
$$
Moreover, if in addition, $n\geq 13$, then $s(n,2n-3)=\frac{n(n-1)}{2}$,
$s(n,2n-4)=\frac{n(n-1)}{2}-1$,
$s(n,2n-5)=\frac{n(n-1)}{2}-\frac{n-1}{2}$, $s(n,0)=0$,
$s(n,1)=\frac{n-1}{2}$, $s(n,2)=n-1$, and $s(n,3)=n$.
\end{thm}

\begin{proof}
$(1)$ Let $G$ be a spanning subgraph of $K_n$ obtained from $k$
edge-disjoint perfect matchings of $K_n$. Clearly,
$|E(G)|=\frac{nk}{2}$ and $\mpo(G)=k$, implying $s(n,k)\leq
\frac{nk}{2}$. Since $\mpo(G)=k$, it follows that $\delta(G)\geq
\mpo(G)\geq k$, and hence $s(n,k)\geq \frac{nk}{2}$. So $s(n,k)=\frac{nk}{2}$.

$(2)$ For odd $n$ and $4\leq k\leq 2n-6$, to show the upper bound,
we let $G$ be a spanning subgraph of $K_n$ derived from
$\min\{\lceil\frac{k}{3}\rceil,\frac{n-1}{2}\}$ edge-disjoint
spanning Hamiltonian cycles of $K_n$. Clearly, $\mpo(G)\geq k$ and
$|E(G)|=\min\{\lceil\frac{k}{3}\rceil,\frac{n-1}{2}\}n$, implying
$s(n,k)\leq \min\{\lceil\frac{k}{3}\rceil,\frac{n-1}{2}\}n$.

We now show the lower bound. Let $G$ be a graph of order $n$ with
$\mpo(G)=k$. Set $V(G)=\{v_1,v_2,\ldots,v_n\}$. For
$v_i,v_j\in V(G)$, we have the following facts.
\begin{itemize}
\item If $v_iv_j \notin E(G)$, then
$d_G(v_i)+d_G(v_j)\geq k$;

\item If $v_iv_j \in E(G)$, then $d_G(v_i)+d_G(v_j)\geq
k+1$.
\end{itemize}
Without loss of generality, let $v_1v_i\in E(G)$ for $2\leq i\leq
x$; $v_1v_i\notin E(G)$ for $x+1\leq i\leq n$. Clearly, $x-1=d_G(v_1)$.
Observe that $d_G(v_1)+d_G(v_i)\geq k+1$ for $2\leq i\leq x$;
$d_G(v_1)+d_G(v_j)\geq k$ for $x+1\leq j\leq n$. Then
$(n-1)d_G(v_1)+\sum_{1\leq i\leq n, \ i\neq 1}d_G(v_i)\geq
(n-1)k+x-1=(n-1)k+d_G(v_1)$. Similarly, we have
\begin{itemize}
\item[] $(n-1)d_G(v_2)+\sum_{1\leq i\leq n, \ i\neq 2}d_G(v_i)\geq (n-1)k+d_G(v_2)$;

\item[] \ \ \ \ $\vdots$

\item[] $(n-1)d_G(v_n)+\sum_{1\leq i\leq n, \ i\neq n}d_G(v_i)\geq (n-1)k+d_G(v_n)$.
\end{itemize}
Then
$$
(n-1)\sum_{1\leq i\leq n}d_G(v_i)+(n-1)\sum_{1\leq i\leq
n}d_G(v_i)\geq n(n-1)k+\sum_{1\leq i\leq n}d_G(v_i),
$$
that is, $2(n-1)\cdot 2|E(G)|\geq n(n-1)k+2|E(G)|$. So $|E(G)|\geq
\frac{n(n-1)k}{4n-6}$, and hence $s(n,k)\geq \frac{n(n-1)k}{4n-6}$,
as desired.

By Proposition \ref{pro3-5}, Theorems \ref{th3-5} and \ref{th3-6} (where we need $n\geq 13$),
we have $s(n,2n-3)=\frac{n(n-1)}{2}$,
$s(n,2n-4)=\frac{n(n-1)}{2}-1$, and
$s(n,2n-5)=\frac{n(n-1)}{2}-\frac{n-1}{2}$. By Lemma \ref{lem4-1},
$s(n,0)=0$, $s(n,1)=\frac{n-1}{2}$, $s(n,2)=n-1$, and $s(n,3)=n$.
\end{proof}

The following observation is immediate.
\begin{obs}
Let $n,k$ be two positive integers. Then $g(n,k)=s(n,k+1)-1$.
\end{obs}

By the above observation, we have the following result for
$g(n,k)$.
\begin{cor}\label{pro4-2}
Let $n,k$ be two positive integers. Then

$(1)$ If $n\geq 4$ is even and $0\leq k\leq n-2$, then
$g(n,k)=\frac{n(k+1)}{2}-1$.

$(2)$ If $n\geq 5$ is odd and $3\leq k\leq 2n-7$, then
$$
\frac{n(n-1)(k+1)}{4n-6}-1 \leq g(n,k)\leq
\min\left\{\left\lceil\frac{k+1}{3}\right\rceil,\frac{n-1}{2}\right\}n-1.
$$
Moreover, if in addition, $n\geq 15$, then, $g(n,2n-3)=\frac{n(n-1)}{2}$;
$g(n,2n-4)=\frac{n(n-1)}{2}-1$; $g(n,2n-5)=\frac{n(n-1)}{2}-2$;
$g(n,2n-6)=\frac{n(n-1)}{2}-\frac{n-1}{2}-1$;
$g(n,0)=\frac{n-1}{2}-1$; $g(n,1)=n-2$; $g(n,2)=n-1$.
\end{cor}

Next, we give the exact value of $f(n,k)$.

\begin{thm}\label{pro4-1}
Let $n,k$ be two positive integers. Then

$(1)$ If $n\geq 2$ is even and $1\leq k\leq n-1$, then
$f(n,k)=\binom{n-1}{2}+k$.

$(2)$ If $n\geq 3$ is odd and $2\leq k\leq 2n-3$, then $f(n,k)=\binom{n-2}{2}+k$.
\end{thm}
\begin{proof}
$(1)$ Let $G$ be a graph with $n$ vertices such that
$|E(G)|\geq \binom{n-1}{2}+k$. Clearly, $|E(\overline{G})|\leq
n-k-1$. Since there are $n-1$ edge-disjoint perfect matchings in
$K_n$, it follows that $G$ contains at least $(n-1)-(n-k-1)=k$
edge-disjoint perfect matchings, and hence $\mpo(G)\geq k$. So
$f(n,k)\leq \binom{n-1}{2}+k$. To show $f(n,k)\geq
\binom{n-1}{2}+k$, we construct $G_k$ as follows: Let $A_k$ be the graph with two components, $K_1$ and $K_{n-k}$, and $B_k$ be $K_{k-1}$; then $G_k$ is obtained by taking $A_k$ and $B_k$, and by adding all possible
edges between the vertices of $A_k$ and the vertices of $B_k$. (In order words, $G_k$ is the join of $A_k$ and $B_k$, that is,
$G_k=K_{k-1}\vee (K_{n-k}\cup K_1)$.)
Clearly, $G_k$ is a connected graph on $n$ vertices,
$|E(G_k)|=\binom{n-1}{2}+k-1$, and $\mpo(G_k)<k$. So
$f(n,k)=\binom{n-1}{2}+k$.

$(2)$ Let $G$ be a graph with $n$ vertices such that
$|E(G)|\geq \binom{n-2}{2}+k$. Clearly, $|E(\overline{G})|\leq
2n-k-3$. For any $X\subseteq E(G)$, $|X|=k-1$, we have
$|E(\overline{G-X})|\leq 2n-4$. Since $\mpo(K_n)=2n-3$, it follows
that $G-X$ has a perfect matching, and hence $\mpo(G)\geq k$. So
$f(n,k)\leq \binom{n-2}{2}+k$. To show $f(n,k)\geq
\binom{n-2}{2}+k$, we let $G_k$ be the graph obtained from $K_{n-2}$
by adding two vertices $u,v$ and the edges in
$E_{G_k}[u,K_{n-2}]\cup E_{G_k}[v,K_{n-2}]$ such that
$|E_{G_k}[u,K_{n-2}]|+|E_{G_k}[v,K_{n-2}]|=k-1$,
$|E_{G_k}[u,K_{n-2}]|\geq 1$, and $|E_{G_k}[v,K_{n-2}]|\geq 1$.
Clearly, $G_k$ is a connected graph on $n$ vertices,
$|E(G_k)|=\binom{n-2}{2}+k-1$, and $\mpo(G_k)<k$. So
$f(n,k)=\binom{n-1}{2}+k$.
\end{proof}

We remark that $f(n,k)$ is relatively large as one can create a Tutte/Berge set as a vertex cut that separates two complete graphs. This is not unusual as the corresponding result for Hamiltonicity has similar characteristics.

\section{Nordhaus-Gaddum-type results}
In this section, we give Nordhaus-Gaddum-type results for matching preclusion number.

\begin{thm}\label{th5-1}
Let $G\in \mathcal {G}(n)$ be a graph. For $n\geq 3$, we have

$(1)$ $$0\leq \mpo(G)+\mpo(\overline{G})\leq \begin{cases}
n-1, &\mbox {\rm if $n$ is even};\\[0.2cm]
2n-3, &\mbox {\rm if $n$ is odd}.
\end{cases}$$

$(2)$ $$0\leq \mpo(G)\cdot \mpo(\overline{G})\leq \begin{cases}
\lceil\frac{n-1}{2}\rceil \lfloor\frac{n-1}{2}\rfloor, &\mbox {\rm if $n$ is even};\\[0.2cm]
(n-2)^2, &\mbox {\rm if $n$ is odd and $n\geq 5$}.
\end{cases}$$
\end{thm}
\begin{proof}
The lower bounds are clear. So we concentrate on the upper bounds.
For $(1)$, if $n$ is even, then we let $\mpo(G)=\ell$. Then $\Delta(G)\geq \delta(G)\geq \ell$, and hence $\mpo(\overline{G})\leq \delta(\overline{G})=n-1-\Delta(G)\leq n-1-\ell$. So
$\mpo(G)+\mpo(\overline{G})\leq n-1$. Clearly, $\mpo(G)+\mpo(\overline{G})\geq 0$. For $(1)$, if $n$ is odd, then we let $u,v$ be two vertices in $G$. Without loss of generality, let $uv\in E(G)$ and $uv\notin E(\overline{G})$. Let $X=E_G[\{u,v\},V(G)-\{u,v\}]\cup \{uv\}$ and $Y=E_{\overline{G}}[\{u,v\},V(G)-\{u,v\}]$. Clearly $|X|+|Y|=2n-3$. Since $u,v$ are two isolated vertices in $G-X$, it follows that $G-X$ contains no
almost perfect matching, and hence $\mpo(G)\leq |X|$. Similarly, since $\overline{G}-Y$ contains no almost perfect matching, it follows that $\mpo(\overline{G})\leq |Y|$. So $\mpo(G)+\mpo(\overline{G})\leq |X|+|Y|=2n-3$.

For $(2)$, the upper bound follow from the upper bound on $\mpo(G)+\mpo(\overline{G})$ from $(1)$ if $n$ is even. (Maximizing $ab$ subject to $a+b=2c$, where $a$ and $b$ are variables and $c$ is a constant, gives an optimal solution at
$a=b=c$.) 
If $n$ is odd, the upper bound can be improved to the one given in the statement.
We consider two cases. We first suppose neither $G$ nor $\overline{G}$ has isolated vertices. Let $u$ be a vertex of $G$ such that $d_G(u)=\delta(G)$. Let $x$ be a vertex of $\overline{G}$ such that $d_{\overline{G}}(x)=\delta(\overline{G})$.
Let $v$ be a neighbor of $u$ and $y$ be a neighbor of $x$. (These vertices exist because $G$ and $\overline{G}$ have no isolated vertices.) Now
\[ \mpo(G)+\mpo(\overline{G})\leq (d_G(u)+d_G(v)-1)+(d_{\overline{G}}(x)+d_{\overline{G}}(y)-1)=\delta(G)+\delta(\overline{G})+d_G(v)+d_{\overline{G}}(y)-2. \]
But $\delta(G)+d_{\overline{G}}(y)\leq d_{G}(y)+d_{\overline{G}}(y)=n-1$ and
$\delta(\overline{G})+d_G(v)\leq d_{\overline{G}}(v)+d_G(v)=n-1$. Thus
$\mpo(G)+\mpo(\overline{G})\leq 2n-4=2(n-2)$. Hence $\mpo(G)\cdot\mpo(\overline{G})\leq (n-2)^2$. Now suppose $G$ has an isolated vertex $w$. Then $\overline{G}$ has no isolated vertices. If $G$ has no edges, then $mp(G)=0$ and the result is clear.
Thus we may assume that $G$ has edges and hence $\overline{G}$ is not complete.
If $G-w$ is complete, then $\overline{G}$ is $K_{1,n-1}$. Thus
$\mpo(\overline{G})=0$ since $n\geq 5$. (If $n=3$, then $\mpo(\overline{G})=2$ and $\mpo(G)=1$.) So we may assume that $G-w$ is not complete.
Consider $H=G-w$. Then $\delta(H)\leq n-3$.
$H$ has $n-1$ vertices and $n-1$ is even. So $\mpo(G)=\mpo(H)\leq \delta(H)$.
Since $\overline{G}$ has no isolated vertices, let $x$ be a vertex of $\overline{G}$ such that $d_{\overline{G}}(x)=\delta(\overline{G})$ and $y$ is a neighbor of $x$. Clearly
$x\neq w$ as $\overline{G}$ is not complete. Now $\delta(H)\leq d_H(x)=d_G(x)\leq n-2$ as $G$ has an isolated vertex. Clearly $d_{\overline{G}}(x)=\delta(\overline{G})\leq n-2$ as $\overline{G}$ is not complete.
If $d_{\overline{G}}(x)\geq 2$. Then we may choose $y\neq w$.
Now $\mpo(\overline{G})\leq d_{\overline{G}}(x)+d_{\overline{G}}(y)-1$. Thus
\[ \mpo(G)+\mpo(\overline{G})\leq \delta(H)+(d_{\overline{G}}(x)+d_{\overline{G}}(y)-1)\leq d_G(y)+d_{\overline{G}}(x)+d_{\overline{G}}(y)-1\leq d_{\overline{G}}(x)+n-2. \]
Since $d_{\overline{G}}(x)\leq n-2$,  $\mpo(G)+\mpo(\overline{G})\leq 2n-4$, and we are done, as before. 
Thus $d_{\overline{G}}(x)=1$ and $x$ is adjacent to $y=w$. Then
\[ \mpo(G)+\mpo(\overline{G})\leq \delta(H)+d_{\overline{G}}(x)+d_{\overline{G}}(y)-1\leq (n-3)+1+(n-1)-1=2n-4, \]
and we are done, as before. 
\end{proof}

We now consider the sharpness of the bounds in Theorem~\ref{th5-1}. We first consider the lower bound. If $n$ is even, we can say more with the following
result.

\begin{pro}\label{pro5-1}
Let $T$ be a tree of order $n\geq 9$. Then
$\mpo(T)=\mpo(\overline{T})=0$ if and only if $n$ is even and $T=K_{1,n-1}$.
\end{pro}

\begin{proof}
If $n$ is even and $T=K_{1,n-1}$, then $\mpo(T)=\mpo(\overline{T})=0$. Conversely, we suppose
$\mpo(T)=\mpo(\overline{T})=0$. We claim that $n$ is even.
Assume, on the contrary, that $n$ is odd.
Since $T$ is a tree, it follows that $\overline{T}$ is a subgraph of $K_n$
by deleting $n-1$ edges. Since $\mpo(K_n)=2n-3>(n-1)$ (by Theorem~\ref{lem2-1}), $\mpo(\overline{T})\geq 1$, which is a contradiction.

We may now assume that $n$ is even. Since the matching preclusion of a path with an odd number of vertices is 2, it follows that
$T$ is not a path. Henceforth we may that $T$ is not a path. We now complete the proof by showing that
$T=K_{1,n-1}$. Assume, on the contrary, that
$T\neq K_{1,n-1}$. Then there exists a vertex $u$ such that $3\leq
d_T(u)\leq n-2$, and hence there exists a vertex $v$, such that $uv
\in E(\overline{T})$. Since $T$ is a tree, it follows that
$\overline{T}-\{u,v\}$ is a spanning subgraph of $K_{n-2}$ by deleting
at most $n-4$ edges, and hence $\overline{G}-\{u,v\}$ has a perfect
matching, say $M$. Clearly, $M\cup
\{uv\}$ is a perfect matching of $T$, and
so $\mpo(\overline{T})\geq 1$, which is a contradiction.
\end{proof}

Proposition~\ref{pro5-1} provides an example to show that the lower bounds in Theorem~\ref{th5-1} are tight for $n$ being even.
If we require both $G$ and $\overline{G}$ to be connected, then we can
consider this example for the lower bounds. This example also works if $n$ is odd.

\noindent {\bf Example 5.1.} For $n\geq 12$, let $G$ be the graph
obtained from $K_{n-5}$ by adding five new vertices
$v_1,v_2,v_3,v_4,v_5$ and then adding the edges in $\{vv_i\,|\,1\leq
i\leq 4\}\cup \{v_5v_1\}$, where $v$ is a vertex of $K_{n-5}$.
Clearly, $G$ and $\overline{G}$ are both connected. Now, 
$\mpo(G)=0$ since deleting $\{v_1,v\}$ from $G$ leaves 3 singletons in the resulting graph, and $\mpo(\overline{G})=0$ since deleting $\{v_1,v_2,v_3,v_4,v_5\}$ from $\overline{G}$ leaves $(n-5)-5\geq 2$ singletons.
(So $G$ and $\overline{G}$ have neither perfect matchings nor almost perfect matchings by Theorem~\ref{thC} and Theorem~\ref{thD}.)

One can easily classify graphs that meet the lower bound for the product.

\begin{obs}\label{obs5-1}
Let $G$ be a graph of order $n$. Then
$\mpo(G)\cdot \mpo(\overline{G})=0$ if and only if $G$ or $\overline{G}$ have no perfect matching or have no almost-perfect matching.
\end{obs}

To show the sharpness of the upper bound for the sum, we consider the following
results.

\begin{thm}\label{th5-2}
If $G$ is an odd graph of order
$n\geq 9$, then $\mpo(G)+\mpo(\overline{G})=2n-3$ if and only if $G=K_n$ or
$\overline{G}=K_n$.
\end{thm}
\begin{proof}
If $G=K_n$ or $\overline{G}=K_n$, then it follows from Proposition \ref{pro3-5} that $\mpo(G)+\mpo(\overline{G})=2n-3$.
Conversely, we suppose $\mpo(G)+\mpo(\overline{G})=2n-3$, $G\neq K_n$ and $\overline{G}\neq K_n$. Then there is a vertex in $G$, say $u$, with
$1\leq d_G(u)\leq n-2$. So we have $u,w,v\in V(G)$ such that $uw\in E(G)$ and $uv\notin E(G)$. We may assume that $d_G(w)\leq d_G(v)$;
otherwise, interchange the roles of $G$ and $\overline{G}$. Now $\mpo(G)+\mpo(\overline{G})\leq
d_G(u)+d_G(w)-1+d_{\overline{G}}(u)+d_{\overline{G}}(v)-1\leq d_G(u)+d_G(v)-1+d_{\overline{G}}(u)+d_{\overline{G}}(v)-1=2n-4<2n-3$, which
is a contradiction.
\end{proof}

If $G$ is even, $K_n$ is still an example to show sharpness of the upper bound for the sum. However, there are other examples such as a cycle of length $4$. Nevertheless, there are some restrictions as the next result shows.

\begin{thm}\label{th5-3}
Let $G$ be an even graph of
order $n\geq 2$. If $\mpo(G)+\mpo(\overline{G})=n-1$, then $G$ is regular.
\end{thm}
\begin{proof}
Suppose $\mpo(G)+\mpo(\overline{G})=n-1$. Assume, on the contrary, that
$G$ is not regular. Then $\Delta(G)>\delta(G)$, and hence there
exist two vertices $u,v$ in $G$ such that $d_G(u)=\Delta(G)$ and
$d_G(v)=\delta(G)$. Since $\mpo(G)\leq d_G(v)$ and
$\mpo(\overline{G})\leq d_{\overline{G}}(u)$, it follows that
$\mpo(G)+\mpo(\overline{G})\leq
d_G(v)+d_{\overline{G}}(u)<d_G(u)+d_{\overline{G}}(u)=n-1$, which is a
contradiction.
\end{proof}


One may wonder where every value between the lower bound and the upper bound of $\mpo(G)+\mpo(\overline{G})$ is achievable, that is, whether an intermediate value theorem exists. We now show that it is true if
$n$ is even. We start with the following observation.

\begin{obs}\label{obs5-2}
If $m$ is odd, then edges of $K_m$ can be decomposed into $m$ edge-disjoint almost perfect matchings where each vertex is missed in exactly one of them.
\end{obs}

\begin{thm}\label{th5-6}
Let $n\geq 4$ be an even number. Then there exists $G$ such that $\mpo(G)+\mpo(\overline{G})=r$ for every $r$ in $[0,n-1]$. Moreover it is realizable with $\mpo(G)=r$ and $\mpo(\overline{G})=0$ for every $r$ in $[0,n-1]$.
\end{thm}
\begin{proof}
Let $n$ be even. Let $r$ be an integer in $[1,n-1]$. Take $r$ edge-disjoint almost perfect matchings from $K_{n-1}$ and consider the graph $H$ induced by them. Add a vertex $v$ that is adjacent to every $n-1$ vertices in $H$. Call the resulting graph $G$.  Then $v$ has degree $n-1$ in $G$. The other vertices has degree $r+1$ or $r$ in $G$. (They have degree $r$ or $r-1$ in $H$.) By construction, $G$ contains $r$ edge-disjoint perfect matchings. So $\mpo(G)=r$. Note that $\overline{G}$ has an isolated vertex, so $\mpo(\overline{G})=0$. This together with the sharpness of the bounds give the result.
\end{proof}

For $n$ is odd, we were not able to obtain an intermediate value theorem although we conjecture it to be true.
We now consider the sharpness of the upper bound for the product. Suppose $n$ is even. Now $K_n$ can be partitioned into $n-1$ edge-disjoint perfect matching $M_1,M_2,\ldots,M_{n-1}$. Let $G_1$ be
the graph induced by the first $\lceil\frac{n-1}{2}\rceil$ $M_i'$s and $G_2$ be the graph induced by the other $\lfloor\frac{n-1}{2}\rfloor$ $M_i'$s. Then $G_2=\overline{G_1}$. Clearly
$\mpo(G_1)=\lceil\frac{n-1}{2}\rceil$ and $\mpo(G_2)=\lfloor\frac{n-1}{2}\rfloor$. The next result shows that such examples must be regular graphs.

\begin{thm}\label{th5-5}
Let $G$ be an even graph of
order $n\geq 2$. If $\mpo(G)\cdot \mpo(\overline{G})=\lceil\frac{n-1}{2}\rceil \lfloor\frac{n-1}{2}\rfloor$, then $G$ is $\lceil\frac{n-1}{2}\rceil$-regular or $\lfloor\frac{n-1}{2}\rfloor$-regular.
\end{thm}
\begin{proof}
Since $\mpo(G)+\mpo(\overline{G})\leq n-1$, If $\mpo(G)\cdot \mpo(\overline{G})=\lceil\frac{n-1}{2}\rceil \lfloor\frac{n-1}{2}\rfloor$, we have $\mpo(G)=\lceil\frac{n-1}{2}\rceil$ and $\mpo(\overline{G})=\lfloor\frac{n-1}{2}\rfloor$ or $\mpo(\overline{G})=\lceil\frac{n-1}{2}\rceil$ and $\mpo(G)=\lfloor\frac{n-1}{2}\rfloor$. Without loss of generality, we consider only $\mpo(G)=\lceil\frac{n-1}{2}\rceil$ and $\mpo(\overline{G})=\lfloor\frac{n-1}{2}\rfloor$. As $\mpo(G)\leq \delta (G)$, we have $\delta (G) \geq \lceil\frac{n-1}{2}\rceil$ and $\delta (\overline{G}) \geq \lfloor\frac{n-1}{2}\rfloor$. Then $G$ is $\lceil\frac{n-1}{2}\rceil$-regular. The other situation can be obtained by the same method.
\end{proof}

If $n$ is odd, then if we let $G=C_5$, the 5-cycle, then $\overline{G}$ is also $C_5$. Thus $\mpo(G)\cdot \mpo(\overline{G})=3^2=(5-2)^2$. However, we do not know examples for infinitely many $n$ to truly show 
that $(n-2)^2$
is tight.

\section{Conclusion}
In this paper, we studied some extremal type problems for the matching preclusion problem. Understanding these types of problems who help researchers in designing reliable and resilient interconnection networks.

\end{document}